\documentclass{kms-c}

\newcommand{\publname}{Commun.~Korean Math.~Soc.}
\issueinfo{}
  {}
  {}
  {}
\pagespan{1}{}
\copyrightinfo{}
  {Korean Mathematical Society}

\usepackage{graphicx}
\usepackage{amssymb}
\allowdisplaybreaks

\theoremstyle{plain}
\newtheorem{theorem}{Theorem}[section]
\newtheorem{proposition}[theorem]{Proposition}
\newtheorem{lemma}[theorem]{Lemma}

\newtheorem{question}[theorem]{Question}

\theoremstyle{definition}
\newtheorem*{definition}{Definition}
\newtheorem{example}[theorem]{Example}

\theoremstyle{remark}

\begin{document}

\title[A weaker notion of the finite factorization property]
{A weaker notion of the finite factorization property}
\author[H. Jiang]{Henry Jiang}
\address{Henry Jiang \\ Detroit Country Day School \\ Beverly Hills, MI 48025}
\email{jiangstem@gmail.com}

\author[S. Kanungo]{Shihan Kanungo}
\address{Shihan Kanungo \\ Henry M. Gunn School \\ Palo Alto, CA 94306}
\email{shihankanungo@gmail.com}
	
\author[H. Kim]{Hwisoo Kim}
\address{Hwisoo Kim \\ Phillips Academy \\ Andover, MA 01810}
\email{hkim25@andover.edu}


\begin{abstract}
An (additive) commutative monoid is called atomic if every given non-invertible element can be written as a sum of atoms (i.e., irreducible elements), in which case, such a sum is called a factorization of the given element. The number of atoms (counting repetitions) in the corresponding sum is called the length of the factorization. Following Geroldinger and Zhong, we say that an atomic monoid $M$ is a length-finite factorization monoid if each $b \in M$ has only finitely many factorizations of any prescribed length. An additive submonoid of $\mathbb{R}_{\ge 0}$ is called a positive monoid. Factorizations in positive monoids have been actively studied in recent years. The main purpose of this paper is to give a better understanding of the non-unique factorization phenomenon in positive monoids through the lens of the length-finite factorization property. To do so, we identify a large class of positive monoids which satisfy the length-finite factorization property. Then we compare the length-finite factorization property to the bounded and the finite factorization properties, which are two properties that have been systematically investigated for more than thirty years.
\end{abstract}

\maketitle

\section{Introduction}
\label{sec:intro}

Following Cohn~\cite{pC68}, we say that a commutative monoid is atomic provided that every non-invertible element can be written as a product of atoms (also called irreducibles), while an integral domain is called atomic provided that its multiplicative monoid is atomic. Following Anderson, Anderson, and Zafrullah~\cite{AAZ90}, we say that a monoid/domain satisfy the finite factorization property if every nonzero nonunit element has only finitely many factorizations (into atoms). A weaker version of the finite factorization property, also introduced in~\cite{AAZ90}, is the bounded factorization property: for each nonzero nonunit element there exists a bound for the number of atoms (counting repetitions) in any factorization of such an element. The bounded and the finite factorization properties have received a great deal of attention since they were introduced back in 1990 (see the recent paper~\cite{AG22} by Anderson and Gotti for a survey on bounded and finite factorization domains).
\smallskip

A commutative monoid is called factorial if every nonunit element has a unique factorization. Factorial monoids are, therefore, atomic. An atomic monoid is called half-factorial if any two factorizations of the same element have the same number of atoms (counting repetitions). The half-factorial property was coined by Zaks~\cite{aZ76} in the context of integral domains, and the same property has been systematically studied since then: see the recent paper~\cite{GLTZ21} by Gao et al. (and references therein) and  see the paper~\cite{CC00} by Chapman and Coykendall for a survey comprising the most relevant studies of half-factoriality until 2000). The following complementary notion of half-factoriality has been considered more recently: a monoid is called length-factorial if any two distinct factorizations of the same element have distinct lengths (it is clear that a monoid is factorial if and only if it is both half-factorial and length-factorial). The notion of length-factoriality was introduced by Coykendall and Smith in~\cite{CS11}, and it has been recently investigated by several authors (see, for instance,~\cite{CCGS21,GZ21,BVZ23}).
\smallskip

We say that the length-factorial property complements the half-factorial property because the former is precisely what a half-factorial monoid needs to be factorial. Similarly, we can define a property that complements the bounded factorization property with respect to the finite factorization property. An atomic monoid is said to satisfy the length-finite factorization property if each element has only finitely many factorizations with any prescribed number of atoms (counting repetitions). Observe that a monoid satisfies the finite factorization property if and only if it satisfies both the bounded and the length-finite factorization properties. The length-finite factorization property was recently introduced by Geroldinger and Zhong in~\cite{GZ21}. The main purpose of this paper is to offer a better understanding of factorizations in positive monoids and positive semirings using the light shed by the length-finite factorization property. While doing so, we provide the first dedicated study of the length-finite factorization property.
\smallskip

A positive monoid is an additive submonoid of $\mathbb{R}$ consisting of nonnegative real numbers. On the other hand, a positive semiring is a positive monoid that contains $1$ and is closed under multiplication. The atomic structure and the arithmetic of factorizations of both positive monoids and positive semirings have been actively studied in recent years. For instance, atomicity and factorizations in positive monoids have been studied by Chapman et al. in~\cite{fG19,CG21a,CG22} and more recently by Gotti and Vulakh in~\cite{GV23}. On the other hand, atomicity and factorizations in positive semirings have been studied by Baeth et al. in~\cite{BG20,BCG21,sZ22} and more recently by Gotti and Polo in~\cite{GP23,GP23a}. Factorizations in positive semirings consisting of rationals have also been considered in recent literature by various authors (see, for instance, \cite{CGG20a,JLZ23}).
\smallskip

In Section~\ref{sec:background}, we briefly discuss the notation, terminology, and main known results we will use throughout the rest of the paper.
\smallskip

Section~\ref{sec:LFF positive monoids}, which is the first section of content, is devoted to study the length-finite factorization property in the setting of positive monoids. The main result we establish in Section~\ref{sec:LFF positive monoids} is Theorem~\ref{thm:main result}, where we identify a large class of positive monoids satisfying the length-finite factorization property. In the second part of the same section, we discuss a few examples connecting the length-finite factorization property with some well-studied atomic properties. It is known that every commutative monoid satisfying the ascending chain on principal ideals (ACCP) is atomic  \cite[Proposition~1.1.4]{GH06}, and it is also known that every monoid that satisfies the bounded factorization property also satisfies the ACCP \cite[Corollary~1]{fHK92}. We provide in Section~\ref{sec:LFF positive monoids} an example of a positive monoid that satisfies the length-finite factorization property but not the ACCP and an example of a positive monoid that satisfies the bounded factorization property but not the length-finite factorization property.
\smallskip

In Section~\ref{sec:LFF positive semirngs}, we consider the length-finite factorization property in the setting of positive semirings. Although the main purpose of Section~\ref{sec:LFF positive semirngs} is to discuss examples of positive semirings in connection with the length-finite factorization property, we also discuss a method to construct positive semirings out of positive monoids using certain exponentiation construction. In this direction, we prove that the ACCP, the bounded factorization property, and the finite factorization property are all preserved under the mentioned exponentiation construction. Additionally, we show exponentiation construction is in fact isomorphic to the monoid semiring over the natural numbers.
It is still unknown to the authors whether the length-finite factorization property is preserved under the same exponentiation construction, and we pose this as Question~\ref{q:does LFF ascend to E(M)}.

\section{Background}
\label{sec:background}
\subsection{General Notation} 

Following standard notation, we let $\mathbb{Z}$, $\mathbb{Q}$, and $\mathbb{R}$ denote the set of integers, rational numbers, and real numbers, respectively. Also, we let $\mathbb{N}$, $\mathbb{N}_0$, and~$\mathbb{P}$ denote the set of positive integers, nonnegative integers, and primes, respectively. 
For $b,c \in \mathbb{Z}$ with $b \le c$, we let $[\![ b,c ]\!]$ denote the set of integers between~$b$ and~$c$, i.e.,
\[
	[\![ b,c ]\!] = \{n \in \mathbb{Z} \mid b \le n \le c\}.
\]
For $S \subseteq \mathbb{R}$ and $r \in \mathbb{R}$, we set $S_{\ge r} = \{s \in S \mid s \ge r\}$ and $S_{> r} = \{s \in S \mid s > r\}$.

\subsection{Commutative Monoids} 

Throughout this paper, we reserve the term {\it monoid} for a cancellative and commutative semigroup with an identity element. Monoids here are written additively unless we specify otherwise. Let $M$ be a monoid. We let $M^\bullet$ denote the set of nonzero elements of $M$, while we let $\mathcal{U}(M)$ denote the group of invertible elements of $M$. The set $M/\mathcal{U}(M) = \{b + \mathcal{U}(M) \mid b \in M\}$ is a monoid under the natural operation induced by the operation of $M$. We say that $M$ is {\it reduced} provided that $\mathcal{U}(M)$ is the trivial group, in which case we can canonically identify $M/\mathcal{U}(M)$ with $M$. A {\it submonoid} of~$M$ is a subset of $M$ that is closed under addition and contains $0$. The monoids we are mostly concerned with in the scope of this paper are positive monoids, which are reduced. A {\it positive monoid} is an additive submonoid of $\mathbb{R}_{\ge 0}$. If $S$ is a subset of $M$, then $\langle S \rangle$ stands for the smallest (under inclusion) submonoid of $M$ that contains $S$. If $M = \langle S \rangle$ for some finite subset $S$, then we say that $M$ is {\it finitely generated}. Additive submonoids of $\mathbb{N}_0$, called {\it numerical monoids}, are always finitely generated. A map $\varphi \colon M \to M'$, where $M'$ is a monoid, is called a ({\it monoid}) {\it homomorphism} if $\varphi(b+c) = \varphi(b) + \varphi(c)$ for all $b,c \in M$.
\smallskip

An element $a \in M \setminus \mathcal{U}(M)$ is called an {\it atom} if for all $b,c \in M$ the fact that $a = b+c$ guarantees that $b \in \mathcal{U}(M)$ or $c \in \mathcal{U}(M)$. We let $\mathcal{A}(M)$ denote the set of all the atoms of $M$. Following~\cite{CDM99}, we say that $M$ is {\it antimatter} if $\mathcal{A}(M)$ is empty. An element $b \in M$ is called {\it atomic} if either $b \in \mathcal{U}(M)$ or if $b$ can be written as a sum of atoms (repetitions are allowed). Following~\cite{pC68}, we say that $M$ is {\it atomic} if every element of $M$ is atomic. Atomic monoids play an important role in this paper. A subset $I$ of $M$ is called an {\it ideal} if $I + M \subseteq I$. An ideal $I$ of $M$ is called {\it principal} if $I = b + M$ for some $b \in M$. A sequence of ideals $(I_n)_{n \ge 1}$ of $M$ is called an {\it ascending chain} if $I_n \subseteq I_{n+1}$ for every $n \in \mathbb{N}$. The monoid $M$ is said to satisfy the {\it ascending chain condition on principal ideals} (ACCP) if every ascending chain of principal ideals $(I_n)_{n \ge 1}$ stabilizes, which means that there exists $N \in \mathbb{N}$ such that $I_n = I_N$ for every $n \ge N$.  If a monoid satisfies the ACCP, then it is atomic \cite[Proposition~1.1.4]{GH06}. The converse does not hold (see, for instance, Example~\ref{ex:LFFM not ACCP}).  In addition, finitely generated monoids satisfy the ACCP (see, for instance, \cite[Proposition~2.7.8]{GH06}), and so they are atomic.

\subsection{Factorizations and Lengths}

Let $\mathsf{Z}(M)$ denote the monoid consisting of all formal sums of atoms in $\mathcal{A}(M/\mathcal{U}(M))$; that is, $\mathsf{Z}(M)$ is the free commutative monoid on $\mathcal{A}(M/\mathcal{U}(M))$. The elements of $\mathsf{Z}(M)$ are called {\it factorizations} in $M$. Let the function $\pi \colon \mathsf{Z}(M) \to M/\mathcal{U}(M)$ be the only monoid homomorphism such that $\pi(a) = a$ for all $a \in \mathcal{A}(M/\mathcal{U}(M))$. Then for each $b \in M$, we set
\[
	\mathsf{Z}(b) := \{z \in \mathsf{Z}(M) \mid \pi(z) = b + \mathcal{U}(M)\},
\]
and call the element of $\mathsf{Z}(b)$ {\it factorizations of} $b$. Observe that $M$ is atomic if and only if $\mathsf{Z}(b)$ is nonempty for every $b \in M$. The monoid $M$ is called a {\it unique factorization monoid} (UFM) if $|\mathsf{Z}(b)| = 1$ for every $b \in M$, while $M$ is called a {\it finite factorization monoid} (FFM) if $1 \le |\mathsf{Z}(b)| < \infty$ for every $b \in M$. It follows from the definitions that every UFM is an FFM. Also, it is well known that every finitely generated monoid is an FFM \cite[Proposition~2.7.8]{GH06}.

If $z := a_1 \cdots a_\ell \in \mathsf{Z}(M)$ for some $a_1, \dots, a_\ell \in \mathcal{A}(M/\mathcal{U}(M))$, then we call $\ell$ the {\it length} of $z$, and we often denote $\ell$ by $|z|$. For $b \in M$, the set
\[
	\mathsf{L}(b) := \{|z| \mid z \in \mathsf{Z}(b)\}
\]
is called the {\it set of lengths} of~$b$. We say that $M$ is a {\it bounded factorization monoid} (BFM) if $1 \le |\mathsf{L}(b)| < \infty$ for every $b \in M$. It follows from the definitions that every FFM is a BFM. Also, it is well known that every BFM satisfies the ACCP \cite[Corollary~1]{fHK92}. For each $b \in M$ and $\ell \in \mathbb{N}$, set
\[
	\mathsf{Z}_\ell(b):=\{z \in \mathsf{Z}(b) \mid |z| = \ell \}.
\]
Following~\cite{GZ21}, we say that $M$ is a {\it length-finite factorization monoid} (LFFM) if $M$ is atomic and $|\mathsf{Z}_\ell(b)| < \infty$ for all $b \in M$ and $\ell \in \mathbb{N}$. From the definitions, we obtain that every FFM is an LFFM. As we shall see in the next section, the notions of a BFM and an LFFM are not comparable.

\subsection{Commutative Semirings}

A nonempty set $S$ endowed with two binary operations `$+$' and `$\cdot$' is called a {\it commutative semiring}, or simply a {\it semiring}\footnote{Although in the scope of this paper we are only interested in semirings that are cancellative and commutative with respect to both operations, it is worth emphasizing that a more general and standard definition of a semiring does not assume any of these conditions (see~\cite{jG99}).} in the scope of this paper, if the following three conditions are satisfied:
\begin{enumerate}
	\item $(S,+)$ is a monoid with its identity element denoted by $0$;
	\smallskip
	
	\item $(S^\bullet, \cdot)$ is a monoid with identity element denoted by $1$;
	\smallskip
	
	\item $b \cdot (c+d)= b \cdot c + b \cdot d$ for all $b, c, d \in S$.
\end{enumerate}
The operations `$+$' and `$\cdot$' are called {\it addition} and {\it multiplication}, respectively. For any $b,c \in S$, we often write $b c$ rather than $b \cdot c$. Let $S$ be a semiring. We refer to the invertible elements of the multiplicative monoid $S^\bullet$ simply as {\it units} of $S$, and we denote the set of units of $S$ by $S^\times$. A subset $S'$ of a semiring $S$ is called a {\it subsemiring} of~$S$ if $(S',+)$ is a submonoid of $(S,+)$ that contains~$1$ and is closed under multiplication. In this paper, we restrict our attention to positive semirings. A {\it positive semiring} is a subsemiring of $\mathbb{R}_{\ge 0}$ under the standard addition and multiplication. Observe that the additive monoid of a positive semiring is a positive monoid.

Let $S$ be a positive semiring. Since the additive monoids of most of the positive semirings we consider in this paper are UFMs, when we mention a monoidal divisibility or atomic property of a semiring, we will be referring to its multiplicative structure. Accordingly, we say that $S$ is {\it antimatter} (resp., {\it atomic}, an {\it LFFS}, a {\it BFS}, an {\it FFS}) provided that the multiplicative monoid $S^\bullet$ is antimatter (resp., atomic, an LFFM, a BFM, an FFM). Similarly, we say that $S$ satisfies the {\it ACCP} if its multiplicative monoid $S^\bullet$ satisfies the ACCP.

\section{Length-Finite Factorization Positive Monoids}
\label{sec:LFF positive monoids}

Following~\cite{GG20}, we say that a positive monoid is {\it decreasing} (resp., {\it increasing}) if it can be generated by a decreasing (resp., increasing) sequence, and we say that a positive monoid is {\it monotone} if it is either decreasing or increasing. Monotone positive monoids were also studied in~\cite{fG19,mB20,BG21,hP22}. It is clear that every numerical monoid is both decreasing and increasing. In general, a positive monoid is both decreasing and increasing if and only if it is finitely generated \cite[Proposition~5.4]{fG19}.

We say that a sequence of real numbers is {\it well-ordered} (resp., {\it co-well-ordered}) if it contains no strictly decreasing (resp., increasing) subsequence. Following Polo~\cite{hP22}, we say that a positive monoid $M$ is {\it well-ordered} (resp., {\it co-well-ordered}) if it can be generated by a well-ordered (resp., co-well-ordered) sequence. It follows from the definitions that every increasing (resp., decreasing) positive monoid is {\it well-ordered} (resp., {\it co-well-ordered}). As the following example illustrates, there are atomic positive monoids that are neither well-ordered nor co-well-ordered.

\begin{example} \label{ex:an atomic PM that is not monotone}
	Let $(p_n)_{n \ge 1}$ be a strictly increasing sequence of primes, and consider the positive monoid
	\[
		M := \left\langle 1 + \frac{(-1)^n}{p_n} \ \Big{|} \ n \in \mathbb{N} \right\rangle.
	\] 
	Set $a_n := 1 + \frac{(-1)^n}{p_n}$ for every $n \in \mathbb{N}$. Since for each $n \in \mathbb{N}$ there is only one defining generator, namely $a_n$, whose $p_n$-adic valuation is negative, the same defining generator must be an atom. As a consequence,
	\[
		\mathcal{A}(M) = \{a_n \mid n \in \mathbb{N}\},
	\]
	which implies that $M$ is atomic. Since $M$ is reduced, any generating set of $M$ must contain $\mathcal{A}(M)$. Therefore the fact that $(a_{2n})_{n \ge 1}$ is a strictly decreasing sequence implies that $M$ cannot contain any well-ordered generating sequence, while the fact that $(a_{2n-1})_{n \ge 1}$ is strictly increasing implies that $M$ cannot contain any co-well-ordered generating sequence. Hence $M$ is neither a well-ordered nor a co-well-ordered positive monoid.
\end{example}
\subsection{A Class of Length-Finite Factorization Monoids}

In this subsection, we identify a large class consisting of positive monoids that are LFFM: indeed, as our main result, we prove that every co-well-ordered positive monoid is an LFFM. In order to prove this, we need the following two lemmas.

\begin{lemma} \label{lem:non-co-well-ordered sums of sequences}
	The sum of finitely many co-well-ordered sequences is a co-well-ordered sequence.
\end{lemma}

\begin{proof}
	It suffices to argue that the sum of two co-well-ordered sequences is again a co-well-ordered sequence as, after proving this we can induct on the number of sequences being added to obtain the statement of the lemma. Let $(a_n)_{n \ge 1}$ and $(b_n)_{n \ge 1}$ be two co-well-ordered sequences. Now suppose, for the sake of a contradiction, that the sequence $(a_n + b_n)_{n \ge 1}$ is not co-well-ordered. Let $(c'_n)_{n \ge 1}$ be a strictly increasing subsequence of $(a_n + b_n)_{n \ge 1}$. Take subsequences $(a'_n)_{n \ge 1}$ and $(b'_n)_{n \ge 1}$ of $(a_n)_{n \ge 1}$ and $(b_n)_{n \ge 1}$, respectively, such that $(c'_n)_{n \ge 1} = (a'_n)_{n \ge 1} + (b'_n)_{n \ge 1}$. In light of the Monotone Subsequence Theorem, after passing to suitable subsequences, we can assume that $(a'_n)_{n \ge 1}$ and $(b'_n)_{n \ge 1}$ are both monotone sequences. Since $(a_n)_{n \ge 1}$ and $(b_n)_{n \ge 1}$ are both co-well-ordered, neither $(a'_n)_{n \ge 1}$ nor $(b'_n)_{n \ge 1}$ can contain a strictly increasing subsequence. Hence there exists $k \in \mathbb{N}$ such that the sequences $(a'_n)_{n \ge k}$ and $(b'_n)_{n \ge k}$ are decreasing. This implies that $(c'_n)_{n \ge k}$ is also a decreasing sequence, which contradicts that $(c'_n)_{n \ge 1}$ is a strictly increasing sequence and, therefore, contains no decreasing subsequences. Hence $(a_n + b_n)_{n \ge 1}$ is a co-well-ordered sequence.
\end{proof}

A pair $(z_1,z_2)$ of factorizations in $\mathsf{Z}(M)$ of the same element of $M$ is called \textit{irredundant} if $z_1$ and $z_2$ do not share any atoms. In addition, we say that a subset $Z$ of $\mathsf{Z}(M)$ is {\it irredundant} if any two distinct factorizations in $Z$ are irredundant.

 \begin{lemma}\label{lem:irredundant set of factorizations}
	 	Let $M$ be an atomic co-well-ordered positive monoid. For all $x \in M$ and $\ell \in \mathbb{N}$, every irredundant subset of $\mathsf{Z}_\ell(x)$ is finite.
	 \end{lemma}

\begin{proof}
	Fix $x \in M$ and $\ell \in \mathbb{N}$. For the sake of a contradiction, suppose that there exists an infinite irredundant subset of $\mathsf{Z}_\ell(x)$. Therefore there exists a sequence of factorizations $(z_n)_{n \ge 1}$ whose terms are pairwise distinct and whose underlying set is an irredundant subset of $\mathsf{Z}_\ell(x)$. For a factorization $z := a_1 \cdots a_\ell \in \mathsf{Z}(M)$ with $a_1, \dots, a_\ell \in \mathcal{A}(M)$ such that $a_1 \le \dots \le a_\ell$ and for each $k \in [\![ 1, \ell ]\!]$, we call $a_k$ the $k$-{\it th atom} of~$z$. For each $n \in \mathbb{N}$ and $k \in [\![ 1, \ell ]\!]$, let $a_n^{(k)}$ denote the $k$-th atom of $z_n$. By virtue of the Monotone Subsequence Theorem, the sequence $( a_n^{(1)} )_{n \ge 1}$ has a monotone subsequence. As a result, after replacing $(z_n)_{n \ge 1}$ by a suitable subsequence, we can assume that $( a_n^{(1)} )_{n \ge 1}$ is monotone. Since the underlying set of $(z_n)_{n \ge 1}$ is irredundant, the sequence $( a_n^{(1)} )_{n \ge 1}$ must be either strictly decreasing or strictly increasing. The fact that $M$ is co-well-ordered, however, ensures that no sequence of atoms of $M$ can be strictly increasing, and so we can assume that $( a_n^{(1)} )_{n \ge 1}$ is strictly decreasing. This, along with the fact that $\sum_{i=1}^\ell a_n^{(i)} = x$ for every $n \in \mathbb{N}$, ensures that the sequence $(s_n)_{n \ge 1} := (a_n^{(2)})_{n \ge 1} + \dots + (a_n^{(\ell)})_{n \ge 1}$ is strictly increasing. Therefore $(s_n)_{n \ge 1}$ is not co-well-ordered, and so it follows from Lemma~\ref{lem:non-co-well-ordered sums of sequences} that $(a_n^{(j)})_{n \ge 1}$ is not co-well-ordered for some $j \in [\![ 2, \ell ]\!]$. Thus, there exists a strictly increasing sequence whose terms are atoms of $M$, which contradicts the fact that $M$ is co-well-ordered. Hence for all $x \in M$ and $\ell \in \mathbb{N}$, every irredundant subset of $\mathsf{Z}_\ell(x)$ must be finite.
\end{proof}

We are in a position to prove the main result of this section.

\begin{theorem} \label{thm:main result}
	Every atomic co-well-ordered positive monoid is an LFFM.
\end{theorem}

\begin{proof}
	Let $M$ be an atomic co-well-ordered positive monoid. Now suppose, by way of contradiction, that $M$ is not an LFFM. Then for some $\ell \in \mathbb{N}$ there exists $x \in M$ such that $|\mathsf{Z}_\ell(x)|  = \infty$. Furthermore, assume that we have chosen $\ell$ as the smallest as it can possibly be. Since $|\mathsf{Z}_1(x)| \le 1$, we see that $\ell \ge 2$. Our next step is to argue the existence of a maximal irredundant subset of $\mathsf{Z}_\ell(x)$ under inclusion.
	
	Let $\mathcal{Z}$ be the poset consisting of all irredundant subsets of $\mathsf{Z}_\ell(x)$ under inclusion. As $\mathsf{Z}_\ell(x)$ is infinite, and so nonempty, $\mathcal{Z}$ is a nonempty poset. In addition, it follows from Lemma~\ref{lem:irredundant set of factorizations} that every element in $\mathcal{Z}$ is a finite set. Now suppose that $\mathcal{C} := \{Z_\gamma \mid \gamma \in \Gamma\}$ is a nonempty chain in the poset $\mathcal{Z}$. We claim that $U := \bigcup_{\gamma \in \Gamma} Z_\gamma$ is an upper bound for $\mathcal{C}$ in $\mathcal{Z}$. If $z$ and $z'$ are two distinct factorizations in $U$, then as $\mathcal{C}$ is a chain, we can pick $\gamma \in \Gamma$ such that $z$ and $z'$ are both contained in $Z_\gamma$. As $Z_\gamma \in \mathcal{Z}$, it must be an irredundant set of factorizations, and so $z$ and $z'$ do not share any atoms. Hence~$U$ is also an irredundant set of factorizations, all of them contained in $\mathsf{Z}_\ell(x)$. As a result, $U \in \mathcal{Z}$, and so $U$ is an upper bound for the chain $\mathcal{C}$ in $\mathcal{Z}$. Now we show that $\mathcal{Z}$ has a maximal element. Pick some element $Z_1\in\mathcal{Z}$. If it is maximal, then we are done. Otherwise we can pick $Z_2\in \mathcal{Z}$ such that $Z_1\subsetneq Z_2$. If $Z_2$ is maximal, we are done. Otherwise we can pick $Z_3\in \mathcal{Z}$ such that $Z_2\subsetneq Z_3$. If $Z_3$ is maximal, we are done. Continuing this process, this process either terminates, in which case we are done, or we have an infinite chain $Z_1\subsetneq Z_2\subsetneq\cdots\subset U$ of factorizations in $\mathcal{Z}$. Since each $Z_i$ adds at least one new element to $U$, this implies that $U$ must have infinitely many elements, contradicting Lemma~\ref{lem:irredundant set of factorizations}. Thus $\mathcal{Z}$ has a maximal element. 
	
	Let $Z$ be a maximal element of $\mathcal{Z}$, that is, an irredundant subset of $\mathsf{Z}_\ell(x)$ that is not properly contained in any other irredundant subset of $\mathsf{Z}_\ell(x)$. Let $A$ be the set consisting of all the atoms of $M$ that appear in at least one factorization in $Z$. Because~$Z$ is finite, so is $A$. The maximality of $Z$ guarantees that every factorization in $\mathsf{Z}_\ell(x)$ shares an atom with at least one factorization in $Z$, and so every factorization in $\mathsf{Z}_\ell(x)$ contains an atom from $A$. As a consequence, after setting
	\[
		Z_a := \{z \in \mathsf{Z}_\ell(x) \mid a \text{ appears in } z\}
	\]
	for each $a \in A$, we obtain that $\mathsf{Z}_\ell(x) = \bigcup_{a \in A} Z_a$. Since $A$ is finite, the fact that $\mathsf{Z}_\ell(x)$ contains infinitely many factorizations ensures the existence of $a \in A$ such that $|Z_a| = \infty$. By definition of $Z_a$, every factorization of $Z_a$ contains the atom $a$. Therefore the set of factorizations $Z'_a := \{z-a \mid z \in Z_a \}$ is a subset of $\mathsf{Z}_{\ell - 1}(x-a)$. Since $|Z_a'| = |Z_a| = \infty$, it follows that $\mathsf{Z}_{\ell - 1}(x-a)$ is infinite, which contradicts the minimality of $\ell$. Hence $M$ is an LFFM.
\end{proof}

The converse of Theorem~\ref{thm:main result} does not hold in general. The following example, which is a continuation of Example~\ref{ex:an atomic PM that is not monotone}, illustrates this observation.

\begin{example}
	Let $(p_n)_{n \ge 1}$ be a strictly increasing sequence of primes, and consider the positive monoid
	\[
		M := \left\langle 1 + \frac{(-1)^n}{p_n} \ \Big{|} \ n \in \mathbb{N} \right\rangle.
	\] 
	We have already seen in Example~\ref{ex:an atomic PM that is not monotone} that $M$ is an atomic positive monoid with $\mathcal{A}(M) = \{a_n \mid n \in \mathbb{N}\}$, where $a_n := 1 + \frac{(-1)^n}{p_n}$ for every $n \in \mathbb{N}$. We have also observed in the same example that $M$ is not co-well-ordered. We claim that $M$ is an FFM and, therefore, an LFFM. To argue this, fix a nonzero $q \in M$. Take $n_q \in \mathbb{N}$ such that $p_k \nmid \mathsf{d}(q)$ for any $k \ge n_q$, and then set $N := \max\{n_q, q+1\}$. Now suppose that the atom $a_k$ appears in a factorization of~$q$. We will argue that $k \le N$. Write $q = \sum_{i=1}^n c_i a_i$ for some $n \in \mathbb{N}_{\ge k}$ and $c_1, \dots, c_n \in \mathbb{N}_0$ such that $c_k \neq 0$. If $p_k \mid \mathsf{d}(q)$, then $k < n_q \le N$. Assume, on the other hand, that $p_k \nmid \mathsf{d}(q)$. This being the case, after applying $p_k$-adic valuations to both sides of the equality $q = \sum_{i=1}^n c_i a_i$, we find that $p_k \mid c_k$. Therefore $q \ge c_k a_k \ge p_k a_k = p_k + (-1)^k$, which implies that $k \le p_k \le q+1 \le N$. Hence $k \le N$, as desired, and so $q$ is divisible in $M$ by only finitely many atoms. Since every element of $M$ is divisible by only finitely many atoms in $M$, it follows from \cite[Theorem~2]{fHK92} that $M$ is an FFM and, in particular, an LFFM.
\end{example}

\medskip
\subsection{Some Concrete Examples}

Although every FFM is an LFFM, the converse of this implication does not hold in general. As the following examples illustrate, there are positive monoids that are LFFM but do not even satisfy the ACCP.

\begin{example} \label{ex:LFFM not ACCP} \hfill
	\begin{enumerate}
		\item[(a)] Let $(p_n)_{n \ge 0}$ be the strictly increasing sequence whose underlying set is $\mathbb{P} \setminus \{2\}$. Then consider the positive monoid defined as follows:
		\[
			M := \left\langle \frac1{2^n p_n} \ \Big{|} \ n \in \mathbb{N}_0 \right\rangle.
		\]
		This monoid is often called the Grams' monoid as it is the essential ingredient used by Grams in her construction of the first atomic domain that does not satisfy the ACCP (see~\cite{aG74} for more details and~\cite{GL23} for a recent generalization). It is well known that $M$ is atomic with
		\[
			\mathcal{A}(M) = \Big\{ \frac1{2^n p_n}  \ \Big{|} \  n \in \mathbb{N}_0 \Big\}
		\]
		but does not satisfy the ACCP: indeed, $(\frac1{2^n} + M )_{n \ge 0}$ is an ascending chain of principal ideals that does not stabilize. Finally, since the sequence of defining generators of $M$ is decreasing, $M$ is a decreasing monoid and, therefore, $M$ is an LFFM by virtue of Theorem~\ref{thm:main result}.
		\smallskip
		
		\item[(b)] Choose $q \in \mathbb{Q} \cap (0,1)$ such that $q^{-1} \notin \mathbb{N}$, and consider the positive monoid $\mathbb{N}_0[q] := \langle q^n \mid n \in \mathbb{N}_0 \rangle$. Using the fact that $q^{-1} \notin \mathbb{N}$, we can argue that $\mathbb{N}_0[q]$ is atomic with $\mathcal{A}(\mathbb{N}_0[q]) = \{q^n \mid n \in \mathbb{N}_0\}$ (this is well known: see \cite[Theorem~6.2]{GG18} and also \cite[Theorem~4.2]{CG22}). On the other hand, $\mathbb{N}_0[q]$ does not satisfy the ACCP because $\mathsf{d}(q) q^n = (\mathsf{d}(q) - \mathsf{n}(q)) q^n + \mathsf{d}(q) q^{n+1}$ for every $n \in \mathbb{N}_0$, where $\mathsf{n}(q)$ and $\mathsf{d}(q)$ denote the numerator and denominator of $q$, respectively. Lastly, note that because the sequence $(q^n)_{n \ge 0}$ is decreasing, Theorem~\ref{thm:main result} ensures that $\mathbb{N}_0[q]$ is an LFFM. Hence the positive monoid $\mathbb{N}_0[q]$ is an LFFM that does not satisfy the ACCP.
	\end{enumerate}
\end{example}

Thus, being an LFFM does not imply satisfying the ACCP and, therefore, not every LFFM is a BFM. Let us proceed to exhibit a positive monoid that is a BFM but not an LFFM.

\begin{example} \label{ex:BFM not LFFM}
	Consider the positive monoid $M := \{0\} \cup \mathbb{Q}_{\ge 1}$. One can readily verify that $\mathcal{A}(M) = \mathbb{Q} \cap [1,2)$ and, therefore, that $M$ is atomic. Since $0$ is not a limit point of $M^\bullet$, it follows from \cite[Proposition~4.5]{fG19} that $M$ is a BFM. In order to argue that $M$ is not an LFFM, it suffices to observe that the element $3$ has infinitely many length-$2$ factorizations in $M$: indeed, the equality $3 = \big( \frac32 - \frac1n \big) + \big( \frac32 + \frac1n \big)$ yields a length-$2$ factorization of $3$ for each $n \in \mathbb{N}$ with $n \ge 3$.
\end{example}

We have seen in Example~\ref{ex:LFFM not ACCP} that being an LFFM is not a sufficient condition for satisfying the ACCP. On the other hand, even if a positive monoid is an LFFM satisfying the ACCP it still may not be a BFM. The following example sheds some light upon this observation.

\begin{example} \label{ex:LFFM and ACCP not BFM}
	Consider the positive monoid $M = \big\langle \frac1p \mid p \in \mathbb{P} \big\rangle$. It is well known that $\mathcal{A}(M) = \big\{ \frac1p \mid p \in \mathbb{P} \big\}$ and also that $M$ satisfies the ACCP (see \cite[Example~3.3]{AG22} or \cite[Proposition~4.2(2)]{fG22} for more details). In addition, since the set of atoms of $M$ is the underlying set of the decreasing sequence $\big( \frac1p \big)_{p \in \mathbb{P}}$, it follows from Theorem~\ref{thm:main result} that $M$ is an LFFM. Finally, observe that $M$ is not a BFM: indeed, $\mathsf{Z}(1) = \big\{p \frac1p \mid p \in \mathbb{P}\}$ and so $\mathsf{L}(1) = \mathbb{P}$. Hence the positive monoid $M$ is an LFFM, satisfies the ACCP, but is not a BFM.
\end{example}

\section{Length-Finite Factorization Positive Semirings}
\label{sec:LFF positive semirngs}
\subsection{A Class of Length-Finite Factorization Positive Semirings}

We can use any positive monoid to construct positive semirings by using certain exponentiation. Let $M$ be a positive monoid. Following~\cite{BCG21}, we set
\[
	E(M) := \big\langle e^m \mid m \in M \big\rangle;
\]
that is, $E(M)$ is the positive monoid generated by the set of positive real numbers $e(M) := \{e^m \mid m \in M\}$. We observe that $1 \in E(M)$ and also that $E(M)$ is closed under multiplication. Hence $E(M)$ is a positive semiring. Also, the fact that $1 = \min E(M)^\bullet$, implies that the multiplicative monoid $E(M)^\bullet$ is reduced, and so $E(M)$ is a reduced positive semiring. Observe that $e(M)$ is a multiplicative monoid that is naturally isomorphic to the positive monoid~$M$.

Suppose now that $M$ consists of algebraic numbers, which is the case we primarily target here. Then it follows from Lindemann-Weierstrass Theorem that the set $e(M)$ is linearly independent over the algebraic numbers (see \cite[Chapter~1]{aB90}), and so the additive monoid $E(M)$ is a free commutative monoid on $e(M)$. Additionally, it follows from \cite[Lemma~2.7]{BCG21}\footnote{Although \cite[Lemma~2.7]{BCG21} is stated for positive monoids, it only applies to positive monoids consisting of algebraic numbers.} that $e(M)$ is a divisor-closed submonoid of the multiplicative monoid of $E(M)$. In this case, every nonzero element $r \in E(M)$ can be written uniquely as $r = c_1 e^{m_1} + \dots + c_k e^{m_k}$ for some coefficients $c_1, \dots, c_k \in \mathbb{N}$ and exponents $m_1, \dots, m_k \in M$ such that $m_1 < \dots < m_k$: as for polynomials, we call $\text{LC}(r) := c_k$ and $\deg r := m_k$ the {\it leading coefficient} and the {\it degree} of $r$, respectively. In addition, we call $S_E(r) := \{m_1,\ldots, m_k\}$ the \textit{exponent set} of $r$. 

It turns out that the properties of satisfying the ACCP, being a BFM, and being an FFM ascend from a positive monoid $M$ to $E(M)$ provided that every element of $M$ is algebraic. We proceed to prove this.

\begin{proposition} \label{prop:ascending properties WRT exponentiation}
	Let $M$ be a positive monoid consisting of algebraic numbers. Then the following statements hold.
	\begin{enumerate}
		\item If $M$ satisfies the ACCP, then $E(M)$ satisfies the ACCP.
		\smallskip
		
		\item If $M$ is a BFM, then $E(M)$ is a BFS.
		\smallskip
		
		\item If $M$ is an FFM, then $E(M)$ is an FFS.
%
	\end{enumerate}
\end{proposition}

\begin{proof}
	(1) Assume that $M$ satisfies the ACCP. Now suppose, by way of contradiction, that $E(M)$ does not satisfy the ACCP. Then there exists an element $s \in E(M)$ and two sequences $(s_n)_{n \ge 1}$ and $(t_n)_{n \ge 1}$ whose terms are nonunits of $E(M)$ such that $s = t_n \prod_{i=1}^n s_i$ for every $n \in \mathbb{N}$. Therefore, for each $n \in \mathbb{N}$,
	\[
		\deg s = \deg t_n + \sum_{i=1}^n \deg s_i \quad \text{ and } \quad \text{LC}(s) = \text{LC}(t_n) \prod_{i=1}^n \text{LC}(s_i).
	\]
	Let $\ell$ be the length of the only factorization of $\text{LC}(s)$ in $\mathbb{N}$. Then at most $\ell$ of the real numbers $\text{LC} (s_1), \dots, \text{LC}(s_n), \text{LC} (t_n)$ can be greater than one. Since the terms $s_k$ and $t_k$ are nonunits for every $k \in \mathbb{N}$, at least $n+1-\ell$ of the real numbers $\deg s_1, \dots, \deg s_n, \deg t_n$ must be nonzero. Thus, at least $n-\ell$ of the real numbers $\deg s_1,\dots,\deg s_n$ must be nonzero. Now consider the ascending chain of principal ideals $\deg t_1+M, \deg t_2+M,\dots$ in $M$. Since at least $n-\ell$ of the real numbers $\deg s_1,\dots,\deg s_n$ must be nonzero for arbitrarily large $n$, and $\deg s_n=\deg t_{n-1}-\deg t_{n}$, this means that this ascending chain of principal ideals must not stabilize. Thus, $M$ does not satisfy the ACCP, a contradiction.
    Thus $E(M)$ satisfies the ACCP.

	\smallskip
	
	(2) Since $M$ is a BFM, it must satisfy the ACCP, and so it follows from part~(1) that $E(M)$ also satisfies the ACCP. Thus, $E(M)$ is atomic. Suppose, by way of contradiction, that $E(M)$ is not a BFS. Then there exists $q\in E(M)$ such that there is no bound to the lengths of factorizations of $q$. Suppose $q=q_1 \cdots q_n$ for nonunits $q_1,\dots,q_n$. Then
	\[
		\deg q = \sum_{i=1}^n \deg q_i \quad \text{ and } \quad \text{LC}(q) = \prod_{i=1}^n \text{LC}(q_i).
	\]
	Let $\ell$ be the length of the factorization $\text{LC}(q)$ in $\mathbb{N}$. Then we see that at most $\ell$ of the real numbers $\text{LC}(q_1),\dots ,\text{LC}(q_n)$ can be greater than one. Since the $q_1, \dots, q_n$ are nonunits, at least $n-\ell$ of them must have degree different from zero. Therefore, after replacing each $\deg q_i$ by one of its factorizations in $M$, we obtain a factorization for $\deg q$ with length at least $n - \ell$. But as $n$ approaches infinity, $n - \ell$ approaches infinity, which means that there is no upper bound to the length of factorizations of $\deg q$ in $M$, contradicting the fact that $M$ is a BFM. Hence $E(M)$ is a BFS.
	\smallskip
	
	(3) Assume that $M$ is an FFM. Since $M$ is a BFM, it follows from (2) that $E(M)$ is a BFS. We claim that, for any positive integers $\ell$ and $q \in E(M)$, there are finitely many ways to choose $a_1,\dots, a_\ell \in \mathcal{A}(E(M))$ such that $q = a_1 \cdots a_\ell$. Since $E(M)$ is a BFS, this would imply that $E(M)$ is also an FFS.
	
	First, we prove that there are finitely many choices for $a_1$. Set $m := m_{1} + \dots + m_{\ell}$, where $m_1 \in S_{E}(a_1), \dots, m_\ell \in S_{E}(a_\ell)$. Observe that $m \in S_E(q)$. Therefore $m_1$ is the sum of a subset of atoms in a factorization of $m$ with respect to monoid $M$. Since $M$ is an FFM, there are finitely many such factorizations. Given any such factorization of length $\ell_m$, there are at most $2^{\ell_m}$ ways to choose $m_1$. Since there are finitely many factorizations, there are finitely many ways to choose $m_1$. Thus, there are finitely many ways to choose the exponent set of $a_1$. 
	
	Now observe that $a_{1} \le q$. Therefore $c \le q$ for every coefficient $c$ of $a_{1}$. If $a_{1}$ has $\ell_{1}$ terms, then there are at most $q^{\ell_{1}}$ ways to choose all the coefficients of $a_{1}$, so there are finitely many ways to choose $a_{1}$. By symmetry, there are also finitely many ways to choose $a_{2},\ldots,a_{\ell}$. Let $N$ be a positive integer such that there are less than $N$ ways to choose each of $a_{1},\ldots,a_{\ell}$. As a result, $q$ has at most $N^{\ell}$ factorizations of length $\ell$. Hence we conclude that $E(M)$ is an FFS.
\end{proof}
We now define another method to construct positive semirings, the monoid semiring of $M$ over the naturals. As seen from the definition, this construction is quite similar to the $E(M)$ construction.
\begin{definition}
    Let $M$ be a positive monoid consisting of algebraic numbers. Then, the \textit{monoid semiring of $M$ over the naturals}, denoted by $\mathbb{N}_0[M]$, is the semiring of polynomials with coefficients in $\mathbb{N}_0$ and exponents in $M$, i.e.
    \[\mathbb{N}_0[M] = \left\{\sum_{i=1}^k c_ix^{m_i}\mid c_i, k\in \mathbb{N}, m_i\in M \right\}.\]
\end{definition}
It turns out that for any positive monoid consisting of algebraic numbers, the semirings $E(M)$ and $\mathbb{N}_0[M]$ are isomorphic.
\begin{proposition} \label{prop:monoid semiring and E(M) are isomorphic}
    Let $M$ be a positive monoid consisting of algebraic numbers. Then, the semirings $E(M)$ and $\mathbb{N}_0[M]$ are isomorphic.
\end{proposition}
\begin{proof}
    Since $M$ consists of algebraic numbers, the set $\{x^{m}\mid m\in M\}$ is linearly independent over the natural numbers. Suppose
    \[
        c_1 x^{m_1} + \dots + c_k x^{m_k} = 0.
    \]
    where $c_1,\ldots, c_k$ are nonzero and $k\ge 1$, and $m_1,\ldots, m_k$ are pairwise distinct.
    Then
    \[
        c_1 e^{m_1}+\dots + c_ke^{m_k}=0,
    \]
    so $e^{m_1},\ldots, e^{m_k}$ are linearly dependent, but this contradicts the Lindemann-Weierstrass Theorem.
    
    Thus, every element of $\mathbb{N}_0[M]$ can be represented uniquely as $s = c_1 x^{m_1} + \dots + c_k x^{m_k}$ for some coefficients $c_1, \dots, c_k \in \mathbb{N}$ and exponents $m_1, \dots, m_k \in M$ such that $m_1 < \dots < m_k$. Then, by defining $\varphi(s) = c_1 e^{m_1} + \dots + c_k e^{m_k}$, we get an isomorphism from $\mathbb{N}_0[M]$ to $E(M)$.
\end{proof}
Proposition~\ref{prop:monoid semiring and E(M) are isomorphic} shows that the results of Proposition~\ref{prop:ascending properties WRT exponentiation} also hold true for $\mathbb{N}_0[M]$.
We can use Proposition~\ref{prop:ascending properties WRT exponentiation} to construct positive semirings satisfying some desired prescribed properties that are easier to achieve with positive monoids. The following example illustrates this observation.

\begin{example} \label{ex:BFS not LFFS}
	Consider the positive monoid $M := \{0\} \cup \mathbb{Q}_{\ge 1}$ (which is also a positive semiring). We have seen in Example~\ref{ex:BFM not LFFM} that $M$ is a BFM with $\mathcal{A}(M) = \mathbb{Q} \cap [1,2)$ and also that $M$ is not an LFFM. Since $M$ is a positive monoid consisting of algebraic numbers, it follows from Proposition~\ref{prop:ascending properties WRT exponentiation} that the exponentiation semiring $E(M)$ is also a BFS. However, also from the fact that $M$ consists of algebraic numbers, we obtain via \cite[Lemma~2.7]{BCG21} that $e^M$ is a divisor-closed submonoid of $E(M)$. Since $e^M$ is isomorphic to $M$, the fact that $M$ is not an LFFM guarantees that $E(M)$ is not an LFFS. Hence $E(M)$ is a positive semiring that is a BFS but not an LFFS.
\end{example}

We have seen in Proposition~\ref{prop:ascending properties WRT exponentiation} that the properties of satisfying the ACCP, having bounded factorizations, and having finite factorizations all ascend from a positive monoid $M$ to its exponentiation semiring $E(M)$ provided that $M$ consists of algebraic elements. We still do not know, however, whether this is also the case for the length-finite factorization property.
\begin{question} \label{q:does LFF ascend to E(M)}
    For a positive monoid $M$ consisting of algebraic numbers, is $E(M)$ an LFFS provided that $M$ is an LFFM? 
\end{question}

\subsection{Further Concrete Examples}

In this final subsection, we provide two further classes of positive semirings. The first class consists of positive semirings that generalize the positive semiring $M$ used in Example~\ref{ex:BFS not LFFS}: this class contains positive semirings that, as the positive semiring $E(M)$ in Example~\ref{ex:BFS not LFFS}, are BFS but not LFFS. The second class consists of positive semirings that are FFS and, therefore, both BFS and LFFS.

\begin{example} \label{ex:conductuve semirings}
	For $r \in \mathbb{R}_{\ge 1}$, consider the positive monoid $S_r := \mathbb{N}_0 \cup \mathbb{R}_{\ge r}$. Because $S_r$ contains $1$ and is closed under multiplication, it is a positive semiring. If $r = 1$, then for each $s \in \mathbb{R}_{> 1}$ the equality $s = \big( s \frac{n}{n+1}\big) \big( \frac{n+1}{n}\big)$ for $n \in \mathbb{N}$ large enough implies that $s$ is not an atom of the multiplicative monoid $S_r^\bullet$ and, therefore, we conclude that $S_r$ is antimatter. For $r > 1$, it follows from \cite[Theorem~5.1]{BCG21} that
	\[
		\mathcal{A}_+(S_r) = \big( \{1\} \cup [r,r+1) \big) \setminus \{ \lceil r \rceil \} \quad \text{ and } \quad \mathcal{A}_\times(S_r) = \big( \mathbb{P}_{< r^2} \cup [r,r^2) \big) \setminus \mathbb{P} \cdot (S_r)_{> 1},
	\]
	and also that both the additive and the multiplicative monoids of $S_r$ are BFMs. Even when $r > 1$, neither the additive nor the multiplicative monoid of $S_r$ may be FFMs (the case $r=2$ was illustrated in \cite[Example~6.4]{BCG21}). We proceed to argue that $S_r$ is not even an LFFS when $r > 1$. To verify that the additive monoid of $S_r$ is not an LFFM, it suffices to observe that the formal sum $(r+\frac{1}{n})+(r+1-\frac{1}{n})$ is a length-2 factorization of $2r+1$ in $(S_r,+)$ for every $n \in \mathbb{N}$ sufficiently large. In a similar way, we can argue that the multiplicative monoid $(S_r^\bullet,\cdot)$ is not an LFFM since the formal product $(s\frac{n+1}{n})(s\frac{n}{n+1})$ for some $s\in (r,r^2)$ is a length-2 factorization of $s^2$ for every sufficiently large $n \in \mathbb{N}$.
\end{example}

The positive monoid in Example~\ref{ex:LFFM not ACCP}(b) is also a positive semiring, and the more general case corresponding to algebraic parameters was studied in \cite{CG22} and, more recently, in~\cite{ABLST23}. However, only the additive structure of these semirings has been systematically investigated.

\begin{example} \label{ex:valuation semirings}
	Let $\alpha \in \mathbb{R}_{> 0}$, and consider the positive monoid $\mathbb{N}_0[\alpha] := \langle \alpha^n \mid n \in \mathbb{N}_0 \rangle$ (the special case when $\alpha$ is rational was briefly considered in Example~\ref{ex:LFFM not ACCP}(b)). Observe that $1 \in \mathbb{N}_0[\alpha]$ and also that $\mathbb{N}_0[\alpha]$ is closed under multiplication. Hence $\mathbb{N}_0[\alpha]$ is a positive semiring. There are choices of $\alpha$ such that the additive monoid $\mathbb{N}_0[\alpha]$ is antimatter (for instance, $\alpha = \frac12$) and there are choices of $\alpha$ such that the additive monoid $\mathbb{N}_0[\alpha]$ is atomic but does not satisfy the ACCP (for instance, $\alpha = \frac23$). In addition,  it follows from \cite[Theorem~4.11]{CG22} that the positive monoid $\mathbb{N}_0[\alpha]$ satisfies the ACCP if and only if it is a BFM if and only if it is an FFM. However, not too much is known about the atomic structure of the multiplicative monoid of $\mathbb{N}_0[\alpha]$.
\end{example}

In connection with Example~\ref{ex:valuation semirings}, we were not able to settle the following question.

\begin{question} \label{q:when does LFF property ascend to N_0[q]}
	For which $q \in \mathbb{Q}_{> 0}$ is the positive semiring $\mathbb{N}_0[q]$ an LFFS?
\end{question}

\section*{Acknowledgments}

We are grateful to our mentors, Prof.\ Jim Coykendall and Dr.\ Felix Gotti, for proposing this project and for their guidance during the preparation of this project. While working on this paper, we were part of PRIMES-USA, a year-long math research program hosted by MIT. We would like to express our collective gratitude to the MIT PRIMES program for arranging such an engaging research experience.

\end{document}